\documentclass[11pt]{amsart}

\usepackage{amsmath,amssymb,amscd}

\usepackage{amsrefs}
\usepackage{enumerate}
\usepackage{xcolor}

\usepackage{amsmath,amssymb}

\usepackage{amsxtra, amsmath}
\usepackage{amssymb, amscd}
\usepackage{graphicx}
\usepackage{mathrsfs}

\usepackage{mathtools}
\usepackage{float}

\newtheorem{thm}{Theorem}[section]
\newtheorem{prop}[thm]{Proposition}

\theoremstyle{definition}
\newtheorem{definition}[thm]{Definition}

\theoremstyle{remark}
\newtheorem{remark}[thm]{Remark}

\usepackage{dsfont}
\allowdisplaybreaks

\title[A $3\times 3$ singular solution to the Matrix Bochner Problem]{A $3\times 3$ singular solution to the Matrix Bochner Problem with $\mathcal{D}(W)$ not of the form $\mathbb{C}[D]$}

\author{Ignacio Bono Parisi}

\thanks{This paper was partially supported by SeCyT-UNC, PIP 33620230100819CB, CONICET, PIP 1220150100356.}

\address{CIEM-FaMAF\\ Universidad Nacional de C\'or\-do\-ba\\
CP 5000, C\'or\-do\-ba,  Argentina}
\email{ignacio.bono@unc.edu.ar}

\begin{document}
\begin{abstract}
    The Matrix Bochner Problem aims to classify weight matrices whose sequences of orthogonal polynomials are eigenfunctions of a second-order differential operator. A major breakthrough in this direction was achieved in \cite{CY18}, where it was shown that, under certain natural conditions on the algebra \(\mathcal{D}(W)\), all solutions arise from Darboux transformations of direct sums of classical scalar weights. In this paper, we study a new \(3 \times 3\) Hermite-type weight matrix and determine its algebra \(\mathcal{D}(W)\) as a \(\mathbb{C}[D_1]\)-module generated by \(\{I, D_2\}\), where \(D_{1}\) and \(D_{2}\) are second-order differential operators. This complete description of the algebra allows us to prove that the weight does not arise from a Darboux transformation of classical scalar weights, showing that it falls outside the classification theorem of \cite{CY18}. Unlike previous examples in \cites{BP23,BP24-1}, which also do not fit within this classification, the algebra \(\mathcal{D}(W)\) of this weight matrix is not generated by a single differential operator $D$, making it a fundamentally different case. These results complement the classification theorem of the Matrix Bochner Problem by providing a new type of singular example.
\end{abstract} 
\maketitle

\section{Introduction}
One of the central problems in the theory of matrix orthogonal polynomials is the 
so-called \emph{Matrix Bochner Problem}. This problem concerns the characterization 
of weight matrices $W(x)$ for which the associated sequence of orthogonal polynomials 
satisfies a second-order matrix differential equation. This problem is particularly relevant due to its deep connections with integrable systems, representation theory, and spectral analysis.

The origin of this problem traces back to the scalar case, where S. Bochner \cite{B29} posed and solved the problem, proving that, up to an affine change of coordinates, the classical weights of Hermite, Laguerre, and Jacobi are the only solutions. Later, A. Durán \cite{D97} extended this 
framework to the matrix setting, leading to the so-called Matrix Bochner Problem. While 
the scalar case is completely understood, its matrix counterpart remains open and presents 
rich mathematical challenges, with significant implications for the study of matrix-valued 
special functions and their applications in mathematical physics and approximation theory.

The first nontrivial solutions of the Matrix Bochner Problem were built from matrix spherical functions \cites{GPT01,GPT02,GPT03a} and direct methods \cite{DG04}. Since then, over the past twenty years, many more examples have been discovered. 
See \cites{GPT05, CG06, DG07, PT07, DdI08, DdI08b, DG05, CMV05, CMV07, P08, PR08, PZ16, KPR12, KRR17}.

\smallskip

A major breakthrough in addressing the Matrix Bochner Problem was achieved by R. Casper and M. Yakimov in \cite{CY18}. They characterized all solutions to the Matrix Bochner Problem that satisfy some natural extra hypotheses on the algebra $\mathcal{D}(W)$ of all differential operators that have a sequence of matrix orthogonal polynomials for $W$ as eigenfunctions. They established the following classification theorem:
\begin{thm}\label{clas1}
(\cite{CY18}, Theorem 1.3). Let $W(x)$ be a weight matrix and suppose that $\mathcal{D}(W)$ contains a $W$-symmetric second-order differential operator
$D = \partial^{2}G_{2}(x) + \partial G_{1}(x) + G_{0}(x)$,
with $G_{2}(x)W(x)$ positive-definite on the support of $W(x)$. 
Then  $W(x)$ is a noncommutative bispectral Darboux transformation of a direct sum of classical weights if and only if the algebra $\mathcal{D}(W)$ is full.
\end{thm}

The algebra $\mathcal{D}(W)$ is said to be full if its module rank coincides with the size of the weight. Equivalently, for a weight matrix $W$ of size $N$, the algebra $\mathcal{D}(W)$ is full if and only if there exist nonzero $W$-symmetric differential operators $D_{1},\ldots,D_{N} \in \mathcal{D}(W)$ such that $D_{i}D_{j} = 0$, for all $i\not=j$, and $D_{1} + \cdots + D_{N}$ is a central element in $\mathcal{D}(W)$ which is not a zero divisor.

These additional hypotheses on the algebra $\mathcal{D}(W)$ are natural enough to encompass almost all known examples in the literature. As a complementary development, recent works \cites{BP23,BP24-1} further explore this setting. In particular, \cite{BP23} analyzes a previously known $2\times 2$ Hermite-type weight and proves that its algebra is a polynomial algebra over a second-order differential operator, hence $\mathcal{D}(W)$ is not full and the weight falls outside the classification. Subsequently, \cite{BP24-1} constructed new examples of arbitrary size $N\times N$ for all three classical types (Hermite, Laguerre, and Jacobi) with the same property. These results show that the algebras $\mathcal{D}(W)$ of these weights do not satisfy the fullness condition and, consequently, do not arise as Darboux transformations of direct sums of classical scalar weights. These new constructions further enrich the landscape of solutions to the Matrix Bochner Problem, complementing the classification of R. Casper and M. Yakimov.

\ 

In this paper, we present a new solution to the Matrix Bochner Problem that falls outside the classification given in \cite{CY18}. This example provides another instance where the algebra $\mathcal{D}(W)$ does not satisfy the fullness condition. Moreover, unlike previously known non-full cases, its algebra is not generated by a single differential operator. This further expands the landscape of solutions, showing that the non-full property is not exclusive to polynomial algebras. 

The solution we present is the $3 \times 3$ Hermite-type weight matrix 
    \[ W(x) = e^{-x^{2}} \begin{psmallmatrix} e^{2bx} + a^{2}x^{2} && ax && acx^{2} \\ ax && 1 && cx \\ acx^{2} && cx && c^{2}x^{2} + 1  \end{psmallmatrix}.\]
 The algebra $\mathcal{D}(W)$ contains two $W$-symmetric second-order differential operators $D_{1}$ and $D_{2}$, as defined in \eqref{D1D2}, that satisfy the relations $D_{1}D_{2} = D_{2}D_{1}$, and $D_{2}D_{1} = D_{2}^{2} + \frac{2(c^{2}+2)}{c^{2}}(D_{1}-D_{2})$. The main theorems of this paper are Theorem \ref{algebra D(W)}, in which we establish that the algebra $\mathcal{D}(W)$ is generated as a $\mathbb{C}[D_{1}]$-module by $\{I,D_{2}\}$, and Theorem \ref{no darboux}, where we prove that $W$ is not a Darboux transformation of any direct sum of classical scalar weights. 

\

The paper is organized as follows. In Section \ref{section 2}, we recall the notions of weight matrices, orthogonal polynomials, the algebra $\mathcal{D}(W)$, the Fourier algebras, the Darboux transformations, and some properties of the classical Hermite polynomials. In Section \ref{section 3}, we present the weight matrix $W$ together with the $W$-symmetric second-order differential operators in the algebra $\mathcal{D}(W)$. We also give an explicit expression of a sequence of orthogonal polynomials for $W$. In Section \ref{section 4}, we study the right Fourier algebra of the weight $W$ by relating it to the right Fourier algebra of the diagonal weight $\tilde{W} = \operatorname{diag}(e^{-x^{2}+2bx},e^{-x^{2}},e^{-x^{2}})$ of scalar Hermite weights. Thus, we obtain the general form of every differential operator in $\mathcal{F}_{R}(W)$. Finally, in Section \ref{section 5}, we determine the algebra $\mathcal{D}(W)$ as a $\mathbb{C}[D_{1}]$-module by using all results in the previous sections. This result allows us to prove that the algebra $\mathcal{D}(W)$ is not full, and consequently, we conclude that $W$ is not a Darboux transformation of any direct sum of classical scalar weights.

\section{Preliminaries} \label{section 2}
\subsection{Orthogonal polynomials and the algebra $\mathcal{D}(W)$}
We say that a matrix-valued smooth function $W:\mathbb{R} \to \operatorname{Mat}_{N}(\mathbb{C})$ is a weight matrix of size $N$ supported on a (possibly unbounded) interval $(x_{0},x_{1})$ if $W(x)$ is Hermitian positive definite almost everywhere on $(x_{0},x_{1})$, $W(x) = 0$ for all $x \notin (x_{0},x_{1})$, and with finite moments of all orders.

Given a weight matrix $W$ we consider the following Hermitian sesquilinear form in $\operatorname{Mat}_{N}(\mathbb{C}[x])$, the algebra of polynomials with coefficients in $\operatorname{Mat}_{N}(\mathbb{C})$, 
$$\langle P , Q \rangle = \langle P , Q \rangle_{W} = \int_{x_{0}}^{x_{1}}P(x)W(x)Q(x)^{\ast} \, dx, \text{ for all } P,Q \in \operatorname{Mat}_{N}(\mathbb{C}[x]).$$

With this sesquilinear form, one can construct sequences 
$\{Q_n\}_{n\in\mathbb{N}_0}$ of matrix-valued orthogonal polynomials for $W$. That is, each $Q_n$ is a polynomial of degree $n$ with a nonsingular leading coefficient, and $\langle Q_n, Q_m \rangle = 0$ for $n \neq m$.
It follows that two sequences of orthogonal polynomials $\{Q_{n}(x)\}_{n\in\mathbb{N}_{0}}$ and $\{R_{n}(x)\}_{n\in \mathbb{N}_{0}}$ for $W$ are related by $R_{n}(x) = M_{n}Q_{n}(x)$ for some sequence of nonsingular constant matrices $M_{n}$.
We observe that there is a unique sequence of monic orthogonal polynomials $\{P_n\}_{n\in\mathbb{N}_0}$ in $\operatorname{Mat}_N(\mathbb{C}[x])$.
By following a standard argument (see \cite{K49} or \cite{K71}) one shows that a sequence of orthogonal polynomials $\{Q_n\}_{n\in\mathbb{N}_0}$ satisfy a three-term recursion relation
\begin{equation}\label{ttrr}
    x Q_n(x)=A_{n}Q_{n+1}(x) + B_{n}Q_{n}(x)+ C_nP_{n-1}(x), \qquad n\in\mathbb{N}_0,
\end{equation}
where $Q_{-1}=0$ and $A_{n}, B_n, C_n$ are matrices depending on $n$ and not on $x$.

Throughout this paper, we consider differential operators of the form $D = \sum_{j=0}^{m} \partial^j F_{j}(x)$, where $\partial = \frac{d}{dx}$ and $F_{j}(x)$ is a matrix-valued function. These operators act {\em on the right-hand side} of a matrix-valued function as follows
$$P(x) \cdot D = \sum_{j=0}^{m} \partial^j(P)(x) F_{j}(x).$$

\noindent 
We consider the algebra of all differential operators with polynomial coefficients 
$$\operatorname{Mat}_{N}(\Omega[x])=\Big\{D = \sum_{j=0}^{m} \partial^{j}F_{j}(x) \, : F_{j} \in \operatorname{Mat}_{N}(\mathbb{C}[x]) \Big \}.$$
\noindent 
More generally, when necessary, we will also consider $\operatorname{Mat}_{N}(\Omega[[x]])$, the set of all differential operators with coefficients in $\mathbb{C}[[x]]$, the ring of power series with complex coefficients.

Given a weight matrix $W$ and  $\{Q_n\}_{n\in \mathbb{N}_0}$ a sequence of matrix-valued orthogonal polynomials with respect to $W$, we consider the algebra 
\begin{equation}\label{algDW}
  \mathcal D(W)=\{D\in \operatorname{Mat}_{N}(\Omega[x])\, : \, Q_n \cdot D=\Gamma_{n} Q_n, \, \text{ with } \Gamma_{n}\in \operatorname{Mat}_N(\mathbb{C}), \text{ for all }n\in\mathbb{N}_0\}.
\end{equation}
The definition of the algebra $\mathcal{D}(W)$ does not depend on the choice of the sequence $Q_{n}$. Since each other sequence is of the form $R_{n}(x) = M_{n}Q_{n}(x)$, we have that $Q_{n}(x)\cdot D = \Gamma_{n}Q_{n}(x)$ if and only if $R_{n}(x) \cdot D = M_{n}\Gamma_{n}M_{n}^{-1}R_{n}(x)$.

\

The following proposition states that the algebra $\mathcal{D}(W)$ is a subalgebra of  
$$\mathcal{D}_{N}(\Omega[x]) = \left \{ D = \sum_{j=0}^{m}\partial^{j}F_{j}(x) \in \operatorname{Mat}_{N}(\Omega[x]) \, : \, \deg(F_{j}) \leq j \right \},$$  
which consists of differential operators $D$ satisfying $\deg(P\cdot D) \leq \deg(P)$ for all polynomial $P \in \operatorname{Mat}_{N}(\mathbb{C}[x])$.  
Moreover, this proposition provides a formula for the eigenvalues of a differential operator for the sequence of monic orthogonal polynomials.  

\begin{prop}[\cite{GT07}, Propositions 2.6 and 2.7]\label{eigenvalue-prop}
  Let $W=W(x)$ be a weight matrix of size $N$ and let $\{P_n\}_{n\geq 0}$ be the sequence of monic orthogonal polynomials for $W$. If $D = \sum_{j=0}^{s}\partial^{i}F_{i}(x)$ is a differential operator of order $s$ such that
  $$P_n\cdot D=\Lambda_n(D) P_n, \qquad \text{for all } n\in\mathbb{N}_0,$$
  with $\Lambda_n(D)\in \operatorname{Mat}_N(\mathbb{C})$, then
 $F_i(x)=\sum_{j=0}^i x^j F_j^i$, $F_j^i \in \operatorname{Mat}_N(\mathbb{C})$, is a polynomial of $\deg(F_i)\leq i$. Moreover, $D$ is determined by the sequence $\{\Lambda_n(D)\}_{n\geq 0}$ and
 \begin{equation}\label{eigenvaluemonicos}
   \Lambda_n(D)=\sum_{i=0}^s [n]_i F_i^i, \qquad \text{for all } n\geq 0,
 \end{equation}
    where $[n]_i=n(n-1)\cdots (n-i+1)$, and  $[n]_0=1$.
\end{prop}

\begin{prop} [\cite{GT07}, Proposition 2.8]\label{prop2.8-GT}
For each $n\in\mathbb{N}_0$, the mapping $D\mapsto \Lambda_n(D) $ is a representation of $\mathcal D(W)$ in $\operatorname{Mat}_N(\mathbb{C})$. 
Moreover, the sequence of representations $\{\Lambda_n\}_{n\in\mathbb{N}_0}$ separates the elements of $\mathcal D(W)$, i.e.
if $\Lambda_n(D_1)=\Lambda_n(D_2)$ for all $n\geq 0$ then $D_1=D_2$. 
\end{prop}

The \textit{formal adjoint} on $\operatorname{Mat}_{N}(\Omega([[x]]))$, denoted by $\mbox{}^*$, is the unique involution that extends the Hermitian conjugation on $\operatorname{Mat}_{N}(\mathbb{C}[x])$ while mapping $\partial I$ to $-\partial I$.  

The \textit{formal $W$-adjoint} of a differential operator $D \in \operatorname{Mat}_{N}(\Omega[x])$ is given by  
\[
D^{\dagger} := W(x) D^{\ast} W(x)^{-1},
\]  
where $D^{\ast}$ represents the formal adjoint of $D$.  

An operator $D \in \operatorname{Mat}_{N}(\Omega[x])$ is said to be \textit{$W$-adjointable} if there exists another operator $\tilde{D} \in \operatorname{Mat}_{N}(\Omega[x])$ satisfying  
\[
\langle P \cdot D, Q \rangle = \langle P, Q \cdot \tilde{D} \rangle,
\]  
for all $P, Q \in \operatorname{Mat}_N(\mathbb{C}[x])$. In this case, we refer to $\tilde{D}$ as the $W$-adjoint of $D$.

\begin{prop} \label{adjunta D(W)} 
    If $D \in \mathcal{D}(W)$, then  $D$ is $W$-adjointable. Moreover, there exists a unique $\tilde D$ in $\mathcal D(W)$ such that 
  $$ \langle P D, Q\rangle  = \langle P, Q\tilde D\rangle \quad \text{ for all } P, Q \in \operatorname{Mat}_{N}(\mathbb{C})[x]. $$
\end{prop}
\begin{proof}
    See Corollary 4.5 in \cite{GT07}.
\end{proof}

\

In this paper, we consider weight matrices $W(x)$ that satisfy certain regularity conditions. More precisely, for every integer $n \geq 0$, the $n$-th derivative $W^{(n)}(x)$ decays exponentially at infinity, and there exists a scalar polynomial $p_n(x)$ such that the product $W^{(n)}(x)p_n(x)$ has finite moments. See \cite{CY18}*{Section 2.2}. Under these assumptions, the following proposition holds.

\begin{prop}[\cite{CY18}, Prop. 2.23] \label{adjuntas}  
If $D \in \operatorname{Mat}_{N}(\Omega[x])$ is $W$-adjointable and the formal $W$-adjoint $D^{\dagger}$ belongs to $\operatorname{Mat}_{N}(\Omega[x])$, then $D^{\dagger}$ is the $W$-adjoint of $D$, i.e.,
 $$\langle \, P\cdot D,Q\, \rangle=\langle \, P,\, Q\cdot {D}^\dagger \rangle,$$ for all $P,Q\in \operatorname{Mat}_N(\mathbb{C}[x])$.
 \end{prop}

\ 

A differential operator $D \in \operatorname{Mat}_{N}(\Omega[x])$ is called $W$-symmetric if $\langle P\cdot D, Q \rangle = \langle P , Q\cdot D \rangle$ for all polynomials $P,Q \in \operatorname{Mat}_{N}(\Omega[x])$. In particular, the algebra $\mathcal{D}(W)$ is determined by the $W$-symmetric operators. We have that the set $\mathcal{S}(W)$, of all $W$-symmetric differential operators in $\mathcal{D}(W)$, satisfies that 
$$\mathcal{D}(W) = \mathcal{S}(W) \oplus i \mathcal{S}(W)$$
as real vector spaces.

\

The following proposition states an important result about the $W$-symmetric differential operators in $\mathcal{D}_{N}(\Omega[x])$. 

\begin{prop}[\cite{GT07}, Prop. 2.10] \label{grados}
    Let $D = \sum_{j=0}^{m}\partial^{j}F_{j} \in \operatorname{Mat}_{N}(\Omega[x])$ be a $W$-symmetric differential operator such that $F_{j}$ is of degree less than or equal to $j$. Then $D$ belongs to the algebra $\mathcal{D}(W)$.
\end{prop}

In particular, to ensure that a second-order differential operator belongs to $\mathcal{S}(W)$, we have the following proposition, 

\begin{prop}\label{symmeq2}
    Let $W$ be a weight matrix supported on $(x_{0},x_{1})$. Let $D = \partial^{2}F_{2}(x) + \partial F_{1}(x) + F_{0} \in \mathcal{D}_{N}(\Omega[x])$ be a second-order differential operator such that
    \begin{equation}
\begin{split}
  F_2 W &=WF_2^*,  \\
    2(F_2W)'-F_1W &=WF_1^*,\\
    (F_2W)''-(F_1W)'+F_0W & =WF_0^*,   
\end{split}
\end{equation}
and 
$$\lim_{x\to x_{0},x_{1}}F_{2}(x)W(x)=0, \quad \lim_{x \to x_{0},x_{1}} (F_{1}(x)W(x)-W(x)F_{1}(x)^{\ast}) = 0.$$
Then $D$ is a $W$-symmetric differential operator in $\mathcal{D}(W)$.
\end{prop}

\subsection{Right Fourier algebra and Darboux transformation}

We recall the concept of the right Fourier algebra associated with a weight matrix $W(x)$, following \cite{CY18}. Consider the space of all semi-infinite sequences of matrix-valued rational functions, defined as  
$$ \mathcal{P} =\{ P:\mathbb{C}\times \mathbb{N}_0 \longrightarrow M_N(\mathbb{C}) \, : \, P(x,n)  \text{ is a rational function of } x \text{ for each fixed } n \}. $$  
We denote by $\operatorname{Mat}_{N}(\mathcal{S})$ the algebra consisting of operators of the form  
\[
\mathscr{M}= \sum_{j=-\ell}^{k} A_{j}(n)\, \delta^j, \quad A_{j}(n)\in\operatorname{Mat}_{N}(\mathbb{C}),
\]  
where the left action on $\mathcal{P}$ is given by  
\begin{equation}\label{discreteop}
(\mathscr{M}\cdot P)(x,n)= \sum_{j=-\ell}^{k} A_{j}(n)(\delta^j \cdot P)(x,n) = \sum_{j=-\ell}^k A_j(n) P(x,n+j).
\end{equation}

Let $P_n(x) = P(x,n)$ be the sequence of monic orthogonal polynomials for the weight matrix $W$. The right Fourier algebra associated with $W$ is defined as  
\begin{equation} \label{Fourier def}
\begin{split}
    \mathcal F_R(W)= \mathcal F_R(P) &=\{ \mathfrak D\in \operatorname{Mat}_{N}(\Omega[x])\, : \, \exists \, \mathscr{M} \in \operatorname{Mat}_{N}(\mathcal{S}) \text{ such that } P(x,n)\cdot \mathfrak D=\mathscr{M}\cdot P(x,n) \}.
\end{split}
\end{equation}
In particular, the algebra $\mathcal{D}(W)$ is a subalgebra of $\mathcal{F}_{R}(W)$.
\smallskip

From \cite{CY18}, Theorem 3.7, we recall an explicit characterization of the right Fourier algebra associated with a weight matrix $W$:  
\begin{equation} \label{fourier algebra}
\mathcal{F}_{R}(W) = \{ \mathfrak{D}\in \operatorname{Mat}_{N}(\Omega[x]) \,:\, \mathfrak{D} \text{ is } W\text{-adjointable and } \mathfrak{D}^{\dagger}\in \operatorname{Mat}_{N}(\Omega[x]) \}.
\end{equation}

\begin{definition}\label{Darbou-transf-def}
Let $W(x)$ and $\widetilde{W}(x)$ be weight matrices, and let $P_n(x)$ and $\widetilde{P}_n(x)$ denote their respective sequences of monic orthogonal polynomials. We say that $\widetilde{P}_n(x)$ is a bispectral Darboux transformation of $P_n(x)$ if there exist differential operators $\mathfrak{D},\widetilde{\mathfrak{D}} \in \mathcal{F}_{R}(W)$, polynomials $F(x),\widetilde{F}(x)$, and sequences of nonsingular matrices $C_n$ and $\widetilde{C}_n$ for almost every $n$, satisfying  
$$C_n\widetilde{P}_n(x) = P_n(x) \cdot \mathfrak{D}F(x)^{-1}, \quad  \widetilde{C}_n P_n(x) = \widetilde{P}_n(x) \cdot \widetilde{F}(x)^{-1}\widetilde{\mathfrak{D}}.$$
We then say that $\widetilde{W}(x)$ is a {\em Darboux transformation} of $W(x)$ whenever $\widetilde{P}_n(x)$ is a bispectral Darboux transformation of $P_n(x)$.  
\end{definition}  

\subsection{The classical Hermite polynomials}
The monic Hermite polynomials $\{h_{n}(x)\}_{n\in\mathbb{N}_{0}}$ are orthogonal with respect to the Hermite weight $w(x) = e^{-x^{2}}$ supported on $\mathbb{R}$. They can be expressed using the Rodrigues formula  
$$h_{n}(x) = \frac{(-1)^{n}}{2^{n}}e^{x^{2}}\frac{d^{n}}{dx^{n}}e^{-x^{2}}.$$  
The shifted monic Hermite polynomials $h_{n}(x-b)$, with $b \in \mathbb{R}$, are obtained from an affine transformation of the argument and remain orthogonal with respect to the shifted Hermite weight $w_{b}(x) = e^{-x^{2}+2bx}$. They satisfy the three-term recurrence relation  
$$h_{n}(x-b)x = h_{n+1}(x-b) + bh_{n}(x-b) + \frac{n}{2}h_{n-1}(x-b),$$  
and are eigenfunctions of the second-order differential operator $\delta_{b} = \partial^{2} + (-2x+2b)\partial$, satisfying  
$$h_{n}(x-b)\cdot \delta_{b} = -2n h_{n}(x-b).$$  
Moreover, \cite{M05} established that the algebra $\mathcal{D}(w_{b})$ is a polynomial algebra over $\delta_{b}$, i.e., $\mathcal{D}(w_{b}) = \mathbb{C}[\delta_{b}]$. Furthermore, the set of the $w_{b}$-symmetric operators in $\mathcal{D}(w_{b})$ is $\mathcal{S}(w_{b}) = \mathbb{R}[\delta_{b}]$.

\smallskip

\section{The $3\times 3$ Hermite weight} \label{section 3}
In this section, we introduce the $3\times 3$ Hermite-type weight, along with the second-order differential operators in $\mathcal{D}(W)$. We also give an expression for a sequence of orthogonal polynomials for $W$.

Let $a,b,c \in \mathbb{R}-\{0\}$. We define the $3\times 3$ Hermite-type weight matrix $W$ as  
\begin{equation}\label{weight}
W(x) = e^{-x^{2}}\begin{pmatrix} e^{2bx} + a^{2}x^{2} && ax && acx^{2} \\ ax && 1 && cx \\ acx^{2} && cx && c^{2}x^{2} + 1  \end{pmatrix}.
\end{equation}
 By Proposition \ref{symmeq2}, we have that the algebra $\mathcal{D}(W)$ contains the following $W$-symmetric second-order differential operators 
\begin{equation}\label{D1D2}
    \begin{split}
        D_{1} & = \partial^{2}I + \partial \begin{pmatrix}  2b-2x && -2bax+2a &&  0 \\ 0 && -2x && 0 \\ 0 && 2c && -2x \end{pmatrix} + \begin{pmatrix} 0 && 0 && 0 \\ 0 && 2 && 0 \\ 0 && 0 && 0 \end{pmatrix}, \text{ and } \\
        D_{2} & = \partial^{2} \begin{pmatrix} 0 && ax && 0 \\ 0 && 1 && 0 \\ 0 && cx && 0 \end{pmatrix} + \partial \begin{pmatrix}  0 && 2a &&  -\frac{2ax}{c} \\ 0 && 0 &&  -\frac{2}{c} \\ 0 &&  2c+\frac{2}{c} && -2x \end{pmatrix} + \begin{pmatrix} 2+\frac{4}{c^{2}} &&  0 && -\frac{2a}{c} \\  0 &&  2+\frac{4}{c^{2}} &&   0 \\  0 && 0 && 0 \end{pmatrix}.
    \end{split}
\end{equation}
Thus, $W$ is a solution to the Matrix Bochner Problem. The eigenvalues of $D_{1}$ and $D_{2}$ are given by

$$\Lambda_{n}(D_{1}) = \begin{pmatrix} -2n && -2ban && 0 \\ 0 && -2n + 2 && 0 \\ 0 && 0 && -2n\end{pmatrix}, \quad \Lambda_{n}(D_{2}) = \begin{pmatrix}2 + \frac{4}{c^{2}} && 0 && -\frac{2a}{c}(n+1) \\ 0 && 2+ \frac{4}{c^{2}} && 0 \\ 0 && 0 && -2n \end{pmatrix}.$$
A result that follows directly from the expression of the eigenvalues of $D_{1}$ and $D_{2}$ is the following proposition.
\begin{prop}\label{relations}  
    The operators $D_{1}$ and $D_{2}$ of the algebra $\mathcal{D}(W)$ satisfy the following relations:  
    \begin{equation*}  
        D_{1}D_{2} = D_{2}D_{1}, \quad D_{2}D_{1} = D_{2}^{2} + \frac{2(c^{2}+2)}{c^{2}}(D_{1}-D_{2}).  
    \end{equation*}  
\end{prop}  

\begin{proof}
    The eigenvalues of the operators satisfy
    \begin{equation*}
        \Lambda_{n}(D_{1})\Lambda_{n}(D_{2}) = \Lambda_{n}(D_{2})\Lambda_{n}(D_{1}), \quad \Lambda_{n}(D_{1})\Lambda_{n}(D_{2}) = \Lambda_{n}(D_{2})^{2}+\frac{2(c^{2}+2)}{c^{2}}(\Lambda_{n}(D_{1})-\Lambda_{n}(D_{2})).
    \end{equation*}
    Thus, the statement holds.
\end{proof}

In order to study the weight matrix $W$, we make the following observation. 
The weight matrix $W$ defined in \eqref{weight} can be factored as
\begin{equation}\label{TwT}
    W(x) = T(x)\tilde{W}(x)T(x)^{\ast},
\end{equation}
where $\tilde{W}(x)$ is the direct sum of scalar Hermite weights $\tilde{W}(x) = \begin{pmatrix} e^{-x^{2}+2bx} && 0 && 0 \\ 0 && e^{-x^{2}} && 0 \\ 0 && 0 && e^{-x^{2}} \end{pmatrix}$, and $T(x) = e^{Ax} = \begin{pmatrix}1 & ax & 0 \\ 0 & 1 & 0 \\ 0 & cx & 1 \end{pmatrix}$, with the nilpotent matrix $A = \begin{pmatrix} 0 & a & 0 \\ 0 & 0 & 0 \\0 & c & 0 \end{pmatrix}$.

With this factorization, from \cite{BP25}, we have an explicit expression for a sequence of orthogonal polynomials for $W$. 
\begin{prop}\label{Qn}
A sequence of orthogonal polynomials for $W$ is given by 
$$Q_{n}(x) = \begin{pmatrix} h_{n}(x-b) && a(h_{n+1}(x)-h_{n}(x-b)x) && 0 \\ -\frac{ae^{-b^{2}}n}{2}h_{n-1}(x-b) && \frac{a^{2}e^{-b^{2}}n}{2}h_{n-1}(x-b)x + h_{n}(x) + \frac{c^{2}n}{2}h_{n-1}(x)x && - \frac{cn}{2}h_{n-1}(x)\\ 0 && c(h_{n+1}(x) - h_{n}(x)x) && h_{n}(x)\end{pmatrix},$$
where $h_{n}(x)$ is the sequence of monic orthogonal polynomials for the Hermite weight $e^{-x^{2}}$.
\end{prop}
The leading coefficient of $Q_{n}(x)$ is 
\begin{equation}\label{An}
    A_{n} = \begin{pmatrix}1 && abn && 0 \\ 0 && \frac{a^{2}e^{-b^{2}}n}{2} + 1 + \frac{c^{2}n}{2} && 0 \\ 0 && 0 && 1\end{pmatrix}.
\end{equation}
By direct calculations, the sequence $Q_{n}$ satisfies the following three-term recurrence relation
\begin{equation*}
    \begin{split}
        Q_{n}(x)x & = \begin{psmallmatrix}1 & - \frac{2ab}{g_{n+1}} & 0 \\ 0 & \frac{g_{n}}{g_{n+1}} & 0 \\ 0 & 0 & 1 \end{psmallmatrix}Q_{n+1}(x) + \begin{psmallmatrix}\frac{b(c^{2}(n+1)+2)}{g_{n+1}} & \frac{a}{g_{n}} & - \frac{ac(n+1)b}{g_{n+1}} \\ \frac{e^{-b^{2}}a}{g_{n+1}} & 0 & \frac{c}{g_{n+1}} \\ 0 & \frac{c}{g_{n}} & 0 \end{psmallmatrix} Q_{n}(x) \\
        & \quad + \begin{psmallmatrix}\frac{n(g_{n}+a^{2}e^{-b^{2}}}{2g_{n}} & 0 & \frac{acn}{2g_{n}} \\ 0 & \frac{n}{2} & 0 \\ \frac{ace^{-b^{2}}n}{2g_{n}} & 0 & \frac{n(g_{n}+c^{2})}{g_{n}} \end{psmallmatrix}Q_{n-1}(x),
    \end{split}
\end{equation*}
where $g_{n} = a^{2}e^{-b^{2}}n+c^{2}n+2$.

\section{The right Fourier algebra for $W$} \label{section 4}
In this section, we study the right Fourier algebra for the weight matrix $W$. By the factorization of $W$ given in \eqref{TwT} we relate the right Fourier algebra of $W$ with the right Fourier algebra of $\tilde{W}$ as we state in the following proposition. 

\begin{prop}\label{tfrt}
    Let $T$ be a matrix-valued polynomial such that $T^{-1}$ is also a matrix-valued polynomial. Let $\tilde{W}$ and $W$ be weight matrices such that $W(x) = T(x)\tilde{W}(x)T(x)^{\ast}$. Then, the right Fourier algebras satisfy 
    $$\mathcal{F}_{R}(W) = T\mathcal{F}_{R}(\tilde{W})T^{-1}.$$
    Moreover, we have that $\mathcal{A}$ is $W$-symmetric if and only if $T^{-1}\mathcal{A}T$ is $\tilde{W}$-symmetric.
\end{prop}
\begin{proof}
    Let $\mathcal{A} \in \mathcal{F}_{R}(W)$, therefore $\mathcal{A}$ is $W$-adjointable, and $\mathcal{A}$ and $\mathcal{A}^{\dagger}$ belong to $\operatorname{Mat}_{N}(\Omega[x])$. Since $T$ and $T^{-1}$ are polynomials, we have that the operator $T^{-1}\mathcal{A}T$ belongs to $\operatorname{Mat}_{N}(\Omega[x])$. Let $P,Q \in \operatorname{Mat}_{N}(\mathbb{C}[x])$, then we have 
    \begin{equation*}
        \begin{split}
            \langle P \cdot T^{-1}\mathcal{A}T, Q \rangle_{\tilde{W}} & = \int_{x_{0}}^{x_{1}}P(x)T^{-1}(x)\cdot \mathcal{A}W(x)(Q(x)T^{-1}(x))^{\ast}dx = \langle PT^{-1}\cdot \mathcal{A},QT^{-1}\rangle_{W}  \\
            & \quad = \langle PT^{-1},QT^{-1}\cdot \mathcal{A}^{\dagger} \rangle_{W} = \langle P , Q \cdot T^{-1} \mathcal{A}^{\dagger}T\rangle_{\tilde{W}}.
        \end{split}
    \end{equation*}
    Thus, $T^{-1}\mathcal{A}T$ is $\tilde{W}$-adjointable, and then it belongs to $\mathcal{F}_{R}(\tilde{W})$. The other inclusion follows analogously. The final part of the statement follows by the fact that the $\tilde{W}$-adjoint of $T^{-1}\mathcal{A}T$ is $T^{-1}\mathcal{A}^{\dagger}T$, where $\mathcal{A}^{\dagger}$ is the $W$-adjoint of $\mathcal{A}$.
\end{proof}

In the next proposition, we study the shape of the operators in $\mathcal{F}_{R}(\tilde{W})$ so we can know the shape of the operators in $\mathcal{F}_{R}(W)$.

\begin{prop}\label{rfw}
    The right Fourier algebra $\mathcal{F}_{R}(\tilde{W})$ satisfies
    $$\mathcal{F}_{R}(\tilde{W}) \subseteq \left \{ \sum_{j=0}^{m}\partial^{j}F_{j} \, \biggm| \, F_{j}(x) = \begin{pmatrix} p_{j}(x) && 0 && 0 \\ 0 && q_{j}(x) && r_{j}(x) \\ 0 && s_{j}(x) && t_{j}(x) \end{pmatrix}, \text{ with } p_{j},q_{j},r_{j},s_{j},t_{j} \in \mathbb{C}[x] \right \}.$$
\end{prop}
\begin{proof}
    Let $\mathcal{A} \in \mathcal{F}_{R}(\tilde{W})$. Then, it follows that $\mathcal{A}$ is $\tilde{W}$-adjointable, it belongs to $\operatorname{Mat}_{N}(\Omega)[x]$, and its adjoint $\mathcal{A}^{\dagger}$ also belongs to $\operatorname{Mat}_{N}(\Omega)[x]$.

    We write $\mathcal{A} = \sum_{j=0}^{m} \partial^{j} F_{j}$, with $F_{j}(x) = \begin{pmatrix} p_{1}^{j}(x) && p_{2}^{j}(x) && p_{3}^{j}(x) \\ p_{4}^{j}(x) && p_{5}^{j}(x) && p_{6}^{j}(x) \\ p_{7}^{j}(x) && p_{8}^{j}(x) && p_{9}^{j}(x) \end{pmatrix} \in \operatorname{Mat}_{N}(\mathbb{C})[x]$. Since $\operatorname{Mat}_{N}(\Omega[x])$ is closed under the formal adjoint $\ast$, it follows that $(\mathcal{A}^{\dagger})^{\ast} = \tilde{W}^{-1}\mathcal{A}\tilde{W}$ belongs to $\operatorname{Mat}_{N}(\Omega[x])$. We have 
    \begin{equation*}
            (\mathcal{A}^{\dagger})^{\ast} =\tilde{W}(x)^{-1}\mathcal{A}\tilde{W}(x) = \sum_{j=0}^{m}\left ( \partial I + \begin{pmatrix} 2(x-b) & 0 & 0 \\ 0 & -2x & 0 \\ 0 & 0 & -2x \end{pmatrix} \right ) ^{j}\tilde{W}(x)^{-1}F_{j}(x)\tilde{W}(x).
    \end{equation*}
    From this, we see that $\tilde{W}(x)^{-1}F_{j}(x)\tilde{W}(x) = \begin{pmatrix}p_{1}^{j}(x) & p_{2}^{j}(x)e^{-2bx} & p_{3}^{j}(x)e^{-2bx} \\ p_{4}^{j}(x)e^{2bx} & p_{5}^{j}(x) & p_{6}^{j}(x) \\ p_{7}^{j}(x)e^{2bx} & p_{8}^{j}(x) & p_{9}^{j}(x) \end{pmatrix}$ must have rational function entries for each $j$. Consequently, $p_{2}^{j} = p_{3}^{j} = p_{4}^{j} = p_{7}^{j} = 0$ for all $j$, and the statement follows.
\end{proof}

Now, we can describe the shape of the differential operators in the right Fourier algebra of $W$. 

\begin{prop}\label{5.3.}
    Let $\mathcal{A}$ be a differential operator in $\mathcal{F}_{R}(W)$, then $\mathcal{A}$ satisfies
    \begin{equation}\label{A}
        \mathcal{A} = \begin{pmatrix} \mathfrak{d}_{1} && -a\mathfrak{d}_{1}x + ax\mathfrak{d}_{2} - acx\mathfrak{d}_{3}x && ax\mathfrak{d}_{3} \\ 0 && \mathfrak{d}_{2} - c\mathfrak{d}_{3}x && \mathfrak{d}_{3} \\ 0 && cx\mathfrak{d}_{2} + \mathfrak{d}_{4} - c^{2} x\mathfrak{d}_{3}x - c\mathfrak{d}_{5}x && cx\mathfrak{d}_{3} + \mathfrak{d}_{5} \end{pmatrix},
    \end{equation}
    for some $\mathfrak{d}_{1}$, $\mathfrak{d}_{2}$, $\mathfrak{d}_{3}$, $\mathfrak{d}_{4}$, and $\mathfrak{d}_{5}\in \Omega[x]$. Moreover, if $\mathcal{A}$ is $W$-symmetric, then $\mathfrak{d}_{1}$ is $w_{b}$-symmetric, and $\mathfrak{d}_{2},\mathfrak{d}_{5}$ are $w$-symmetric, where $w_{b}(x) = e^{-x^{2}+2bx}$ and $w(x) = e^{-x^{2}}$.
\end{prop}
\begin{proof}
    By Proposition \ref{tfrt}, we have that $\mathcal{A} = T\mathcal{B}T^{-1}$ for some differential operator $\mathcal{B} \in \mathcal{F}_{R}(\tilde{W})$. Now, by Proposition \ref{rfw}, we have that 
    $$\mathcal{B} = \begin{pmatrix}\mathfrak{d}_{1} && 0 && 0 \\ 0 && \mathfrak{d}_{2} && \mathfrak{d}_{3} \\ 0 && \mathfrak{d}_{4} && \mathfrak{d}_{5} \end{pmatrix}$$
    for some $\mathfrak{d}_{1}$, $\mathfrak{d}_{2}$, $\mathfrak{d}_{3}$, $\mathfrak{d}_{4}$, and $\mathfrak{d}_{5} \in \Omega[x]$. We also have by Proposition \ref{tfrt} that if $\mathcal{A}$ is $W$-symmetric, then $\mathcal{B}$ is $\tilde{W}$-symmetric. Thus, we have $\mathcal{B} = \tilde{W}\mathcal{B}^{\ast}\tilde{W}^{-1}$, which implies that $\mathfrak{d}_{1} = w_{b}\mathfrak{d}_{1}^{\ast}w_{b}^{-1}$, $\mathfrak{d}_{2} = w\mathfrak{d}_{2}^{\ast}w^{-1}$, and $\mathfrak{d}_{5} = w\mathfrak{d}_{5}^{\ast}w^{-1}$. Finally, by computing $T\mathcal{B}T^{-1}$, the statement holds.
\end{proof}

\section{The algebra $\mathcal{D}(W)$} \label{section 5}
This section aims to prove that the algebra $\mathcal{D}(W)$ is generated by the $W$-symmetric second-order differential operators $D_{1}$ and $D_{2}$ defined in \eqref{D1D2}. First of all, we give the explicit factorization of $D_{1}$ and $D_{2}$ as $T\mathcal{B}T^{-1}$ for a $\tilde{W}$-symmetric operator $\mathcal{B}$,
\begin{equation}\label{fact fact}
    \begin{split}
        D_{1} & = T(x) \begin{pmatrix} \partial^{2} + \partial (-2x+2b) & 0 & 0 \\ 0 & \partial^{2} + \partial (-2x) + 2 & 0 \\ 0 & 0 & \partial^{2} + \partial (-2x) \end{pmatrix} T(x)^{-1}, \\
        D_{2} & = T(x) \begin{pmatrix} \frac{2(c^{2}+2)}{c^{2}} & 0 & 0 \\ 0 & \partial^{2} + \partial (-2x) +\frac{2(c^{2}+2)}{c^{2}} & -\partial \frac{2}{c} \\0 & \frac{2}{c} (\partial -2x) & 2\end{pmatrix}T(x)^{-1}.
    \end{split}
\end{equation}

\begin{prop}\label{5.1.}
    Let $D$ be a $W$-symmetric differential operator in $\mathcal{D}(W)$. Then, $D$ has the form
    \begin{equation}\label{AD}
        D = \begin{pmatrix} \mathfrak{d}_{1} & -a\mathfrak{d}_{1}x + ax\mathfrak{d}_{2} - acx\mathfrak{d}_{3}x & ax\mathfrak{d}_{3} \\ 0 & \mathfrak{d}_{2} - c \mathfrak{d}_{3}x & \mathfrak{d}_{3} \\ 0 & cx\mathfrak{d}_{2} + \mathfrak{d}_{4} - c^{2}x\mathfrak{d}_{3}x - c \mathfrak{d}_{5}x & cx\mathfrak{d}_{3} + \mathfrak{d}_{5} \end{pmatrix},
    \end{equation}
    for some $\mathfrak{d}_{1} \in \mathcal{S}(e^{-x^{2}+2bx})$, $\mathfrak{d}_{2}, \mathfrak{d}_{5} \in \mathcal{S}(e^{-x^{2}})$, and $\mathfrak{d}_{3}, \mathfrak{d}_{4} \in \Omega[x]$. Moreover, the eigenvalues of $D$ are given by  
    \[
    \Lambda_{n}(D) = \begin{pmatrix} \gamma_{1}(n) & \gamma_{2}(n) & \gamma_{3}(n) \\ 0 & \gamma_{4}(n) & 0 \\ 0 & \gamma_{5}(n) & \gamma_{6}(n) \end{pmatrix},
    \]
    for some polynomials $\gamma_{j}(n)$, $j = 1,\ldots,6$.
\end{prop}

\begin{proof}
    Since $\mathcal{D}(W)$ is a subalgebra of $\mathcal{F}_{R}(W)$, Proposition \ref{5.3.} implies that $D$ has the same expression as \eqref{A} for some differential operators $\mathfrak{d}_{j} \in \Omega[x]$, $j = 1,\ldots,5$.  
    Given that $D$ is $W$-symmetric, it follows that $\mathfrak{d}_{1}$ is $w_{b}$-symmetric, and $\mathfrak{d}_{2}$ and $\mathfrak{d}_{5}$ are $w$-symmetric, where $w_{b}(x) = e^{-x^{2}+2bx}$ and $w(x) = e^{-x^{2}}$.  

    By Proposition \ref{eigenvalue-prop}, each entry of $D$ is a differential operator in $\mathcal{D}_{1}(\Omega[x])$. Therefore, $\mathfrak{d}_{1}$ belongs to $\mathcal{D}_{1}(\Omega[x])$. Since the $(1,3)$ entry of $D$ implies that $ax\mathfrak{d}_{3} \in \mathcal{D}_{1}(\Omega[x])$, it follows that $\mathfrak{d}_{2}$ and $\mathfrak{d}_{5}$ belong to $\mathcal{D}_{1}(\Omega[x])$ as well.  

    Therefore, by Proposition \ref{grados}, we conclude that $\mathfrak{d}_{1} \in \mathcal{D}(w_{b})$, $\mathfrak{d}_{2}$ and $\mathfrak{d}_{5} \in \mathcal{D}(w)$. We write $\mathfrak{d}_{3} = \sum_{j=0}^{m}\partial^{j}f_{j}(x)$, with $f_{j}$ polynomials. Since $ax\mathfrak{d}_{3}$ belongs to $\mathcal{D}_{1}(\Omega[x])$, it follows that $\deg(f_{j})<j$. Thus, the final claim follows immediately from the formula for $\Lambda_{n}(D)$ given in Proposition \ref{eigenvalue-prop}.  
\end{proof}

The previous proposition establishes a strong result regarding the structure of the $W$-symmetric differential operators in $\mathcal{D}(W)$. We will use these results to prove the following theorem, one of the main results of this paper, which determines the algebra $\mathcal{D}(W)$ as a module over $\mathbb{C}[D_{1}]$. The proof consists of simplifying a given $W$-symmetric operator in $\mathcal{D}(W)$ by eliminating certain entries, leveraging the structural information provided by the previous proposition and the factored forms of the operators $D_1$ and $D_2$ introduced at the beginning of the section in \eqref{fact fact}. We will also use the fact that, since $D_1$ commutes with $D_2$, any operator of the form $p(D_1)D_2$ is $W$-symmetric for any polynomial $p\in \mathbb{R}[x]$.

\begin{thm}\label{algebra D(W)}
    The algebra $\mathcal{D}(W)$ is generated as a $\mathbb{C}[D_{1}]$-module by $\{I, D_2\}$.
\end{thm}
\begin{proof}
    Given a $W$-symmetric differential operator $D \in \mathcal{D}(W)$, by Proposition \ref{5.1.}, we know that $D$ has the form given in \eqref{AD}, with $\mathfrak{d}_{1}$ an operator in $\mathcal{S}(e^{-x^{2}+2bx})$. Therefore, $\mathfrak{d}_{1}$ is equal to  $p_{1}(\delta_{b})$ for some polynomial $p_{1} \in \mathbb{R}[x]$, and $\delta_{b} = \partial^{2} + \partial (-2x+2b)$. Then, by the expression of $D_{1}$ given in \eqref{fact fact}, we have that the $W$-symmetric differential operator $D-p_{1}(D_{1})$ belongs to $\mathcal{D}(W)$, and has the form given in \eqref{AD} with $\mathfrak{d}_{1} = 0$ and $\mathfrak{d}_{5} = p_{2}(\delta)$ for some polynomial $p_{2}\in \mathbb{R}[x]$, with $\delta = \partial^{2} + \partial (-2x)$. Now, we consider the $W$-symmetric operator $E = D-p_{1}(D_{1})-p_{2}(D_{1})(-\frac{c^{2}}{4})(D_{2}-\frac{2(c^{2}+2)}{c^{2}}I)\in \mathcal{D}(W)$. It follows that $E$ has the form given in \eqref{AD} with $\mathfrak{d}_{1}=0$ and $\mathfrak{d}_{5} = 0$, 
    $$E = T(x) \begin{pmatrix} 0 & 0 & 0 \\ 0 & \mathfrak{d}_{2} & \mathfrak{d}_{3} \\ 0 & \mathfrak{d}_{4} & 0\end{pmatrix}T(x)^{-1} =\begin{pmatrix} 0 & ax\mathfrak{d}_{2} - acx\mathfrak{d}_{3}x & ax\mathfrak{d}_{3} \\ 0 & \mathfrak{d}_{2} - c\mathfrak{d}_{3}x & \mathfrak{d}_{3} \\ 0 & cx\mathfrak{d}_{2} + \mathfrak{d}_{4} - c^{2}c\mathfrak{d}_{3}x & cx\mathfrak{d}_{3}\end{pmatrix}.$$
    The eigenvalues of $E$ for the monic orthogonal polynomials are given by
    $$\Lambda_{n}(E) = \begin{pmatrix} 0 & \gamma_{1}(n) & \gamma_{2}(n) \\ 0 & \gamma_{3}(n) & 0 \\0 & \gamma_{4}(n) & \gamma_{5}(n)\end{pmatrix},$$
    where $\gamma_{1},\gamma_{2},\gamma_{3},\gamma_{4},\gamma_{5}$ are polynomials.
    Then, we have that $A_{n}^{-1}Q_{n}(x) \cdot E = \Lambda_{n}(E)A_{n}^{-1}Q_{n}(x)$, where $Q_{n}$ is the sequence of orthogonal polynomials for $W$ given in Proposition \ref{Qn}, and $A_{n}$ is the sequence given in \eqref{An}. This implies that the following equation holds
    \begin{equation*}
        Q_{n}(x)T(x)\cdot \begin{pmatrix} 0 & 0 & 0 \\ 0 & \mathfrak{d}_{2} & \mathfrak{d}_{3} \\ 0 & \mathfrak{d}_{4} & 0 \end{pmatrix} = A_{n}\begin{pmatrix} 0 & \gamma_{1}(n) & \gamma_{2}(n) \\ 0 & \gamma_{3}(n) & 0 \\ 0 & \gamma_{4}(n) & \gamma_{5}(n) \end{pmatrix}A_{n}^{-1}Q_{n}(x)T(x).
    \end{equation*}
From the entry $(2,1)$ and $(3,1)$ of the above equation, we obtain that $\gamma_{3}(n) = 0$ and $\gamma_{4}(n) = 0$. Now, from the entry $(1,1)$, we obtain that $\gamma_{1}(n) = 0$. Then, from the entry $(2,3)$ we have that $\mathfrak{d}_{3} = 0$, which implies that $\gamma_{2}(n) = 0$ and $\gamma_{5}(n) = 0$. Thus, $\Lambda_{n}(E) = 0$ and we conclude that $D = p_{1}(D_{1}) + p_{2}(D_{1})(-\frac{c^{2}}{4})(D_{2}-\frac{2(c^{2}+2)}{c^{2}}I)$. 
\end{proof}
\begin{remark}
As a direct consequence of the above theorem, the algebra $\mathcal{D}(W)$ is also generated as an algebra by $\{I, D_{1}, D_{2} \}$. And the generators satisfy the relations given in Proposition \ref{relations}. In particular, $\mathcal{D}(W)$ is a commutative algebra.
\end{remark}
Now, we are in a position to establish the following main result.

\begin{thm}\label{no darboux}
    The weight matrix $W$ cannot be obtained as a Darboux transformation of a direct sum of classical scalar weights.
\end{thm}
\begin{proof}
    We have that $W$ is a Darboux transformation of a direct sum of classical scalar weights if and only if the algebra $\mathcal{D}(W)$ is full. This is equivalent to the existence of nonzero $W$-symmetric differential operators $E_1, E_2$, and $E_3$ such that $E_i E_j = 0$ for all $i \neq j$, and $E_1 + E_2 + E_3$ is an element of the center of $\mathcal{D}(W)$ that is not a zero divisor. 

By Theorem \ref{algebra D(W)}, the leading coefficient of any $W$-symmetric differential operator $E$ has the form  
\[
\begin{pmatrix} 
\alpha & \beta a x & 0 \\ 
0 & \alpha + \beta & 0 \\ 
0 & \beta c x & \alpha 
\end{pmatrix},
\]
for some $\alpha, \beta \in \mathbb{R}$. 

Thus, the condition $E_i E_j = 0$ for all $i \neq j$ holds if and only if at least one of the operators $E_i$ is zero, which implies that $\mathcal{D}(W)$ is not full.
\end{proof}

\begin{remark}\label{geometric}
The proof of Theorem \ref{no darboux} can be rephrased in geometric terms. 
From the relations of the generators of $\mathcal{D}(W)$ in Proposition \ref{relations} and by the first isomorphism theorem, we obtain that $\mathcal{D}(W)$ is isomorphic to the quotient ring $R = \mathbb{C}[u,v]/(uv)$. 
Here the isomorphism is the one induced by
\[
\varphi:\mathbb{C}[u,v]\longrightarrow \mathcal{D}(W), 
\quad u \mapsto D_{1}-D_{2}, \quad v \mapsto D_{2} - \tfrac{2(c^{2}+2)}{c^{2}}.
\]
The existence of the operators $E_{1},E_{2},$ and $E_{3}$ as in the proof of Theorem \ref{no darboux} would then imply that $R$ admits at least three minimal prime ideals, whereas $R$ has exactly two. 
\end{remark}

\section*{Acknowledgments}

The author is grateful to the anonymous referee for improving the paper by suggesting a simplification of the proof of Proposition \ref{rfw} and a geometric interpretation of the proof of Theorem \ref{no darboux}, which has been included in Remark \ref{geometric}.

\

\bibliographystyle{amsplain} 
\bibliography{referencias} 

\end{document}